\documentclass[11pt,reqno]{amsart}
\usepackage{amsfonts}
\usepackage{ifthen}
\usepackage{amsthm}
\usepackage{amsmath}
\usepackage{graphicx}
\usepackage{caption}
\usepackage{subcaption}
\usepackage{amscd,amssymb}
\usepackage{color}
\usepackage{hyperref}
\usepackage{array,multirow,makecell}

\usepackage{cite,epstopdf}

\setcellgapes{1pt}
\makegapedcells
\newcolumntype{R}[1]{>{\raggedleft\arraybackslash }b{#1}}
\newcolumntype{L}[1]{>{\raggedright\arraybackslash }b{#1}}
\newcolumntype{C}[1]{>{\centering\arraybackslash }b{#1}}
\newcounter{minutes}
\divide\time by 60
\newcounter{hours}
\multiply\time by 60 \addtocounter{minutes}{-\time}

\newtheorem{theorem}{Theorem}
\newtheorem{lemma}{Lemma}

\newtheorem{corollary}{Corollary}

\numberwithin{equation}{section}

\title[Subordination Involving Gauss Hypergeometric Function]{Subordination Involving Gauss Hypergeometric Function}

\author[A. Kumar and S. Das]
{Anish Kumar and Sourav Das}

\address{{\bf Anish Kumar}\newline
Department of Mathematics,
National Institute of
Technology Jamshedpur,\newline
Jamshedpur 831014, Jharkhand, India}
\email{ak8107690@gmail.com}

\address{{\bf Sourav Das}\newline
Department of Mathematics,
National Institute
of Technology Jamshedpur,\newline
Jamshedpur 831014, Jharkhand, India}
\email{souravdasmath@gmail.com, souravdas.math@nitjsr.ac.in}

\begin{document}

\begin{abstract}
The primary objective of this work is to obtain some sufficient conditions so that normalized Gauss hypergeometric function satisfies exponential starlikeness and convexity in the unit disk. Moreover, conditions on parameter of this function has been derived for being Janowski convexity and starlikeness with the help of differential subprdination.  Results established in this work are presumably new and their significance is illustrated by several consequences.
\end{abstract}
\maketitle

\section{Introduction and Motivation}
In the last few years, generalized hypergeometric functions have sparked great interest among researchers. Mainly used by L. de. Branges \cite{debrang} in 1985 into the proof of Melin conjecture which promted into the Bieberbach conjecture. A few year ago to this proof only some literature deals with the univalent function theory for different kind of Gauss hypergeometric function, confluent hypergeometric function, other generalized hypergeometric function like Bessel function, Struve function, Lommel function using various methods. In \cite{sim, yagmur, orhan, orhan1, owa, noreen, Ali, baricz}, researchers derived several require conditions on the parameters for these special functions lies on geometric characteristics such as convexity, starlikeness, close-to-convexity and uniformly convex. Exponential starlikeness, convexity and also janowski starlikeness and convexity have been established for confluent hypergeometric function  in  \cite{Ali, Naz}. Mocanu and Miller investigated one of the probaly earliest manuscript to demonstrate specific characteristics such as convexity and starlikeness for this functions in 1990 \cite{ssm}. Carath\'eodory result are shown for Gaussian hypergeometric function by using methodology of differential subordination theory \cite{oros}. In this article, various sufficient conditions have been obtained on the involved parameters in Gauss hypergeometric function $F$ which satisfies geometric properties such as exponential starlikeness and convexity, Janowski starlikeness and convexity employing differential subordination theory.

Let $\mathcal{A}$ denote the normalized collection of analytic functions, which satisfy the condition $F^{\prime}(0)-1=F(0)=0$. A function $F\in \mathcal{A}$ which is univalent in the unit disk such that $U=\frac{zF^{\prime}}{F}$ or   $U=1+\frac{zF^{\prime\prime}}{F^{\prime}}$ lies in the domain $|\log U|<1$ of the right half plane connected with the exponential function. Subordination play an important character in defining these geometric properties. If $g$ and $h$ are two analytic function in $\mathbb{D}=\{z\in \mathbb{C}:|z|<1\}$, then $g$ is subordinate to $h$ if there exists a Schwartz function $l$ in $\mathbb{D}$ such that $g=h o l$ in $\mathbb{D}$. It can be denoted also as $g\prec h$ if $g(0)=h(0)$ and $g(D)\subseteq h(D)$. The Class $\mathcal{P}_{e}$ contains analytic function $p$  in $\mathbb{D}$ with $p(z) \prec e^{z}$ and $p(0)=1$ for every $z \in \mathbb{D}$. A function $F\in \mathcal{A}$ is called exponential convex (or starlike) if $1+\frac{zF^{\prime\prime}(z)}{F^{\prime}(z)}$ or $\frac{zF^{\prime(z)}}{F(z)}$ belong to class $\mathcal{P}_{e}$. It is denoted as $\mathcal{K}_{e}$ and $\mathcal{S}_{e}$. Ma and Minda \cite{Maminda} introduced particular cases of subclass of starlike $S^{*}$ and convex function $\mathcal{K}^{*}$. Mathematical charecterization of  $S^{*}$ and $\mathcal{K}^{*}$ is defined as follows:
$$\mathcal{S}^*=\left\{F\in \mathcal{A}:  \Re\left(\frac{zF^\prime(z)}{F(z)}\right)>0 \quad
\mbox{for all} \; z\in\mathbb{D}\right\},$$
$$\mathcal{K}=\left\{F\in \mathcal{A}:  \Re\left(1+\frac{zF^{\prime\prime}(z)}{F^{\prime}(z)}\right)>0 \quad
\forall \; z\in\mathbb{D}\right\}.$$ 

A function $F\in \mathcal{A}$ is lemniscate convex if $1+\frac{zF^{\prime\prime}(z)}{F^{\prime}(z)}$ lie in bounded region by right half of lemniscate of Bernouli $\{k:|k^{2}-1=1\}$. In terms of subordination lemniscate starlike and lemniscate convexity is defined as $\frac{zF^\prime(z)}{F(z)}\prec \sqrt{1+z}$ and $1+\frac{zF^{\prime\prime}(z)}{F(z)}\prec \sqrt{1+z}$ respectively. For $-1\leq D<C\leq 1$, suppose that $P[C,D]$ be the class containing normalized analytic function $p(z)=1+c_{1}(z)+\cdots $ in $\mathbb{D}$ satisfying $p(z)\prec \frac{1+Cz}{1+Dz}$. For $0\leq \beta<1$, then $P[1-2\beta,-1]$ is the class of function $1+c_{1}(z)+\cdots$ holds $\Re(p(z))>\beta$ in $\mathbb{D}$. The class $S^{*}[C,D]$ of Janowski Starlike function \cite{janowski} and convexity if $\mathbb{F}\in \mathcal{A}$ satisfy $\frac{z\mathbb{F}^{\prime}(z)}{F(z)}\in P[C,D]$ and $1+\frac{z\mathbb{F}^{\prime\prime}(z)}{\mathbb{F}(z)}\in P[C,D]$ respectively.

Assume $w\neq 0,-1,-2,\cdots$ then function $$F(u,v;w;z)=\sum_{n=0}^{\infty}\frac{(u)_{n}(v)_{n}}{(w)_{n}}\frac{z^{n}}{n!}$$ is said to be Gauss hypergeometric function is analytic in $\mathbb{D}$ and satisfy the differential equations $z(1-z)F^{\prime\prime}+[w-(u+v+1)z]F^{\prime}-uvF=0$. $F$ holds some recurrence relation \cite{Miller} as:
\begin{align*}
    &F(u,v;w;z)=F(v,u;w;z),\\
    &wF^{\prime}(u,v;w;z)=uvF(u+1,v+1;w+1;z),\\
    &F(u,v;w;z)=(1-z)^{w-u-v}F(w-u,w-v;w;z).
\end{align*}

  In the year 2009, the~following inequality was established by
Pog\'{a}ny and  \linebreak {Srivastava} (\cite{PS}, p. 133, Theorem 4)
\begin{equation}\label{fi}
\psi_0\; \exp\left(\frac{\psi_1}{\psi_0}\;|z|\right)
\leqq {}_p\Psi_q\left[
\begin{array}{rr}
{(a_j,A_j)_{j=1,\cdots,p}};\\
\\
{(b_j,B_j)_{j=1,\cdots,q}};
\end{array}\;z
\right]
\leqq \psi_0-\left(1-e^{|z|}\right)\psi_1
\end{equation}
for all suitably restricted $z, a_j,A_j,b_\ell,B_\ell\in \mathbb{R}
\; (j=1,\cdots,p;\; \ell=1,\cdots,q)$ and
for all ${}_p\Psi_q[z]$ satisfying
the following inequalities:
$$\psi_1>\psi_2\;\;\textrm{and}\;\;\psi_1^2<\psi_0\;\psi_2,$$
where
$$\psi_k:=\frac{\prod\limits_{j=1}^p\Gamma(a_j+A_j\; k)}
{\prod\limits_{j=1}^q\Gamma(b_j+B_j\; k)}\qquad (k=0,1,2).$$
It will be helpful to obtain many results regarding geometric characteristics in geometric function theory.

The paper is systemized as follows. In Section \ref{sec1}, we have mentioned some useful lemmas, which will be helpful to derive our key results. Section \ref{sec2} discusses some sufficient conditions so that the Gauss hypergeometric function possesses such as exponential starlikeness and convexity. Moreover, some consequences and  important remarks have been shown in this section. In Section \ref{sec3}, we have considered normalized Gauss hypergeometric function for which Janowski starlikeness and convexity have been discussed.  Consequence and remarks related to this have also been shown in this section. 

\section{Useful Lemmas}\label{sec1}
  Some Lemmas have been recollected in the below section, which will be useful to show our main results.



 \begin{lemma}\rm{\cite{Medirata}}\label{Subordination Lemma}
Let $S_{1}$ be the image set of the function $f(z)$ and $S_{2}$ be the image set of the function $e^z$ then $S_{1} \subset  S_{2}$ if $$|f(z)-1| < 1-\frac{1}{e},$$ and $f(0)=1$,
 where e is the Euler's number.
\end{lemma}

\begin{lemma}\rm{\cite{Miller,mocamu}}\label{lem2}
 Suppose that $\Omega \subset \mathbb{C}$ and   $\psi:\mathbb{C}^{3}\times \mathbb{D}\rightarrow \mathbb{C}$ hold $\psi(i\rho,\sigma,\mu+iv;z)\notin {\Omega}$ when $z\in \mathbb{D}$, $\rho$ real $\sigma\leq -\frac{1+\rho^{2}}{2}$ and $\sigma+\mu\leq 0$. If $q$ is analytic in $\mathbb{D}$ with $q(0)=1$ and   $\psi(q(z),zq^{\prime}(z),z^{2}q^{\prime\prime}(z);z)\in \Omega$ for $z\in \mathbb{D}$, then $\Re({q(z)})>0$ in $\mathbb{D}$.
\end{lemma}
In this case $\psi:\mathbb{C}^{2}\times \mathbb{D}\rightarrow \mathbb{C}$, then the condition Lemma \ref{lem2} convert into 

$\psi(i\rho,\sigma;z)\notin \Omega$ for $\rho$ real and $\sigma\leq -\frac{1+\rho^{2}}{2}$.
\section{Exponential Starlikeness and Convexity of Gauss Hypergeometic fucntion}\label{sec2}
In this section, we consider the Gauss hypergeometric function and study some of its geometric properties which satisfy the criteria of exponential subordination. 
\begin{theorem}
Assume that the parameters $u,v\in \mathbb{R}$, $w\neq 0,-1,-2,\cdots$ be constrained and u,v and w also fulfill the conditions:
\begin{displaymath}
{\rm (H_1):}\left\{ \begin{array}{ll}
{\rm (i)} & 2(u+2)(v+2)<3(w+2);\\
{\rm (ii)} & \frac{(u+1)(v+1)(w+2)}{(u+2)(v+2)(w+1)}<\frac{2}{3};\\
{\rm (iii)} &\left|{\frac{
u(u+1)v(v+1)}{2(w)(w+1)}+\frac{uv}
{w(e-1)}}\right|< \frac{1}{e},
\end{array} \right.
\end{displaymath} then
$F(u,v;w;z)\in \mathcal{P}_{e}$ w.r.t domain of unit disk.
\end{theorem}
\begin{proof}
To prove the above theorem, we will use the Lemma \ref{Subordination Lemma} which gives us a sufficient condition to guarantee that a function $F(u,v;w,z)$ is contained in the image set of exponential function under the domain of unit disk.
Consider the Gauss hypergeometric function $F(u,v;w,z)=\sum_{n=0}^{\infty}\frac{(u)_{n}(v)_{n}}{(w)_{n}}\frac{z^{n}}{n!}$ which satisfy

$$F(u,v;w,0)=1.$$
Now, it suffices to show that $$|F(u,v;w,z)-1| < 1-\frac{1}{e}.$$
we have,
\begin{align}\label{eqsub}
   \left|\sum_{n=1}^{\infty}\frac{(u)_{n}(v)_{n}}{(w)_{n}}\frac{z^{n}}{n!}\right|&<\left|\frac{\Gamma(w)}{\Gamma(u)\Gamma(v)}\sum_{n=1}^{\infty}\frac{\Gamma(u+n)\Gamma(v+n)}{\Gamma(w+n)n!}\right|\nonumber \\
   &=\left|\frac{\Gamma(w)}{\Gamma(u)\Gamma(v)}\sum_{n=0}^{\infty}\frac{\Gamma(u+1+n)\Gamma(v+1+n)\Gamma(n+1)}{\Gamma(w+1+n)\Gamma(n+2)n!}\right|\nonumber\\
   &=\left|\frac{\Gamma(w)}{\Gamma(u)\Gamma(v)}\right|\left|{}_3\Psi_2\left[
\begin{array}{rr}
(u+1,1),(v+1,1),(1,1);\\
\\
(w+1,1),(2,1);
\end{array}\;1\right]\right|.
\end{align}
 In this case, we have
$$\psi_0=\frac{\Gamma(u+1)\Gamma(v+1)}{\Gamma(w+1)},\;\;
\psi_1=\frac{\Gamma(u+2)\Gamma(v+2)}{2\Gamma(w+2)}\;\;\;\textrm{and}
\;\;\;\psi_2=\frac{\Gamma(u+3)\Gamma(v+3)}{3\Gamma(w+3)}.$$

It can be noted that the given hypothesis equivalent to $\psi_2<\psi_1$
and $\psi_1^2<\psi_0\psi_2.$ Therefore, by~(\ref{fi}),
we find that
\begin{equation}\label{MM2}
{}_3\Psi_2\left[
\begin{array}{rr}
(u+1,1),(v+1,1),(1,1);\\
\\
(w+1,1),(2,1);
\end{array}\;1\right]
\leq \frac{\Gamma(u+2)\Gamma(v+2)(e-1)}{2\Gamma(w+2)}+\frac{\Gamma(u+1)\Gamma(v+1)}{\Gamma(w+1)}.
\end{equation}
Using \eqref{eqsub}, we get
\begin{align*}
 & \left|\frac{\Gamma(w)}{\Gamma(u)\Gamma(v)}{}_3\Psi_2\left[
\begin{array}{rr}
(u+1,1),(v+1,1),(1,1);\\
\\
(w+1,1),(2,1);
\end{array}\;1\right]\right| \\
 &\leq  \left|\frac{\Gamma(w)}{\Gamma(u)\Gamma(v)}\right|\left|\frac{\Gamma(u+2)\Gamma(v+2)(e-1)}{2\Gamma(w+2)}+\frac{\Gamma(u+1)\Gamma(v+1)}{\Gamma(w+1)}\right|,\\
 & \implies |F(u,v;w,z)-1| < 1-\frac{1}{e},
\end{align*}
by using the condition ${\rm (H_1):} {\rm (iii)}$. Furthermore, with~the help of Lemma \ref{Subordination Lemma},
we get the required~result.
\end{proof}
Now, let us consider normalized Gauss hypergeometric function $\mathbb{F}(u,v;w;z)$ to prove the next result as:
$$\mathbb{F}(u,v;w;z)=\sum_{n=0}^{\infty}\frac{(u)_{n}(v)_{n}}{(w)_{n}}\frac{z^{n+1}}{n!}.$$
In the below theorem sufficient condition on parameter $u,v$ and $w$ have been established such that Gauss hypergeometric function $F(u,v;w;z)$ is exponential convex.

\begin{theorem}
Suppose that the parameters $u,v\in \mathbb{R}$, $w\neq 0,-1,-2,\cdots$ be constrained and u,v and w also hold the conditions:
\begin{displaymath}
{\rm (H_2):}\left\{ \begin{array}{ll}
{\rm (i)} & 3(u+1)(v+1)<2(w+1);\\
{\rm (ii)} & 4(w+1)uv<3(u+1)(v+1)w;\\
{\rm (iii)} & 4(u+2)(v+2)<2(w+2);\\
{\rm (iv)} & 8(w+2)(u+1)(v+1)<9(u+2)(v+2)(w+1);\\
{\rm (v)} &\left|\frac{uv}{w}+\frac{1}{e-1}+{\frac{
3u(u+1)v(v+1)(e-1)}{e(w)(w+1)}+\frac{2uv}
{we}}\right|< \frac{1}{e},
\end{array} \right.
\end{displaymath} then
$\mathbb{F}(u,v;w;z)\in \mathcal{K}_{e}$ w.r.t domain of unit disk.
\end{theorem}
\begin{proof}
To show the required result it is sufficient to show that
$$\left|\frac{z\mathbb{F}^{\prime\prime}(u,v;w;z)}{\mathbb{F}^{\prime}(u,v;w;z)}\right| < 1-\frac{1}{e}.$$
We  have,
\begin{align}\label{eqthm2.1}
|z\mathbb{F}^{\prime\prime}(u,v;w;z)|&=\left|\sum_{n=0}^{\infty}\frac{n(n+1)(u)_{n}(v)_{n}}{(w)_{n}}\frac{z^{n}}{n!}\right|\nonumber\\
 &<\left|\frac{\Gamma(w)}{\Gamma(u)\Gamma(v)}\sum_{n=0}^{\infty}\frac{\Gamma(u+n)\Gamma(v+n)\Gamma(n+2)}{\Gamma(w+n)\Gamma(n+1)n!}\right|\nonumber\\
   &=\left|\frac{\Gamma(w)}{\Gamma(u)\Gamma(v)}\right|\left|{}_3\Psi_2\left[
\begin{array}{rr}
(u,1),(v,1),(2,1);\\
\\
(w,1),(1,1);
\end{array}\;1\right]\right|.
\end{align}
 In this case, we have
$$\psi_0=\frac{\Gamma(u)\Gamma(v)}{\Gamma(w)},\;\;
\psi_1=\frac{2\Gamma(u+1)\Gamma(v+1)}{\Gamma(w+1)}\;\;\;\textrm{and}
\;\;\;\psi_2=\frac{3\Gamma(u+2)\Gamma(v+2)}{\Gamma(w+2)}.$$

It can be noted that the given hypothesis equivalent to $\psi_2<\psi_1$
and $\psi_1^2<\psi_0\psi_2.$ Therefore, by~(\ref{fi}),
we find that
\begin{equation}\label{MM2}
{}_3\Psi_2\left[
\begin{array}{rr}
(u,1),(v,1),(2,1);\\
\\
(w,1),(1,1);
\end{array}\;1\right]
\leq \frac{\Gamma(u+1)\Gamma(v+1)(e-1)}{2\Gamma(w+1)}+\frac{\Gamma(u)\Gamma(v)}{\Gamma(w)}.
\end{equation}
Now, we compute with~the help of the inequality:
$|z_1+z_2|\geq\big||z_1|-|z_2|\big|,$ 
\begin{align}\label{eqthm2.2}
    |\mathbb{F}^{\prime}(z)|&=\left|\sum_{n=0}^{\infty}\frac{(n+1)(u)_{n}(v)_{n}}{(w)_{n}}\frac{z^{n}}{n!}\right|\nonumber\\
    &> 1-\left|\sum_{n=0}^{\infty}\frac{(n+2)(u)_{n+1}(v)_{n+1}}{(w)_{n+1}}\frac{1}{(n+1)!}\right|\nonumber\\
    &> 1-\left|\frac{\Gamma(w)}{\Gamma(u)\Gamma(v)}\sum_{n=0}^{\infty}\frac{\Gamma(n+3)\Gamma(u+n+1)\Gamma(v+n+1)}{\Gamma(w+n+1)\Gamma(n+2)}\frac{1}{n!}\right|\nonumber\\
    &=1-\left|\frac{\Gamma(w)}{\Gamma(u)\Gamma(v)}\right|\left|{}_3\Psi_2\left[
\begin{array}{rr}
(u+1,1),(v+1,1),(3,1);\\
\\
(w+1,1),(2,1);
\end{array}\;1\right]\right|
\end{align}
 From above case, we get
$${\psi}^{'}_0=\frac{2\Gamma(u+1)\Gamma(v+1)}{\Gamma(w+1)},\;\;
{\psi}^{'}_1=\frac{3\Gamma(u+2)\Gamma(v+2)}{\Gamma(w+2)}\;\;\;\textrm{and}
\;\;\;{\psi}^{'}_2=\frac{4\Gamma(u+3)\Gamma(v+3)}{\Gamma(w+3)}.$$
It can be notify that the given hypothesis equivalent to ${\psi}^{'}_2<{\psi}^{'}_1$
and ${{\psi}^{'}_1}^{2}<{\psi}^{'}_0{\psi}^{'}_2.$ Therefore, by~(\ref{fi}),
we find that
\begin{equation}\label{MM2.1}
{}_3\Psi_2\left[
\begin{array}{rr}
(u+1,1),(v+1,1),(3,1);\\
\\
(w+1,1),(2,1);
\end{array}\;1\right]
\leq \frac{3\Gamma(u+2)\Gamma(v+2)(e-1)}{\Gamma(w+2)}+\frac{2\Gamma(u+1)\Gamma(v+1)}{\Gamma(w+1)}.
\end{equation}
Now equations \eqref{eqthm2.1}, \eqref{MM2}, \eqref{eqthm2.2} and \eqref{MM2.1} have been taken into account for proceeding the proof as:
\begin{align*}
    \left|\frac{z\mathbb{F}^{\prime\prime}(u,v;w;z)}{\mathbb{F}^{\prime}(u,v;w;z)}\right|&< \frac{\left|\frac{\Gamma(w)}{\Gamma(u)\Gamma(v)}\left(\frac{\Gamma(u+1)\Gamma(v+1)(e-1)}{2\Gamma(w+1)}+\frac{\Gamma(u)\Gamma(v)}{\Gamma(w)}\right)\right|}{1-\left|\frac{\Gamma(w)}{\Gamma(u)\Gamma(v)}\left(\frac{3\Gamma(u+2)\Gamma(v+2)(e-1)}{\Gamma(w+2)}+\frac{2\Gamma(u+1)\Gamma(v+1)}{\Gamma(w+1)}\right)\right|}
\end{align*}
with the help of condition ${\rm (H_2):} {\rm (v)}$ and  Furthermore, Lemma \ref{Subordination Lemma} taking into account,
we get the desired~result.
\end{proof}


Applying the duality theorem between two classes $\mathcal{S}_{e}$ and $\mathcal{K}_{e}$, which state as : $F\in \mathcal{K}_{e}$ if and only if $zF^{\prime}\in \mathcal{S}_{e}$. we have the property $wF^{\prime}(u,v;w;z)=uvF(u+1,v+1;w+1;z)$, sufficient conditions for $zF(u,v;w;z)$ is exponential starlike can be obtained.
\begin{corollary}
Let the parameters $u,v\in \mathbb{R}$ and $w\neq 0,-1,-2,-3,\cdots$  be constrained and u,v and w satisfy the condition of Theorem $(2)$ as $u=u-1,v=v-1.w=w-1$, then the function $zF(u,v;w;z)\in \mathcal{S}_{e}$ or $zF(u,v;w;z)$
is exponential starlike function.
\end{corollary}
In the next section, we have obtained sufficient conditions such that  Gauss hypergeometric function $\phi(u,v;w;z)$ belong to Janowski convexity. Further we will show that $z\phi(u,v;w;z)$ belong to Janowski starlikeness $S^{*}[C,D]$.

\section{Janowski starlikeness and convexity of Gauss hypergeometric function}\label{sec3}
Sufficient conditions has been entertained in this section  using differential subordination so that Gauss hypergeometric function satisfy Janowski starlikeness and convexity.
\begin{theorem}\label{thm3}
    Let parameters $u,v$ and $w\neq 0,-1,-2,\cdots$ be constrained such that it satisfy the conditions  for $-1\leq D<C\leq 1$,
    \begin{align}\label{eq1.3}
    &1+C-D+w(1+D)\nonumber\\
    &-\left|1+C-D+(c+d+2)(1+D)+\frac{(u+v+uv+1)(D+1)^{2}}{C-D}\right|>0,
    \end{align}
    further
    \begin{align}\label{eq1.4}
        &\left(2+\frac{w(1+D)}{2}+\frac{w(D-1)}{2}-\left((D-C)-(u+v+2)D+\frac{(u+v+uv+1)(1-D^{2})}{D-C}\right)^{2}\right)\nonumber\\&2\left((D-C)-(u+v+2)D-\frac{(u+v+uv+1)(1-D^{2})}{C-D}\right)^{2}+\nonumber\\&(\frac{(u+v+uv+1)}{2(C-D)}(1-D)^{2}-1-\frac{(u+v+2)(1+D)}{2}-\frac{(u+v+uv+1)(D+1)^{2}}{2(D-C)}+\nonumber\\
&\left(-1
-\frac{[(u+v+2)]}{2}(D-1)\right))^{2}\geq 0,
    \end{align}
    whenever,
    \begin{align}\label{eq1.5}
        &\left(\frac{(u+v+uv+1)}{2(C-D)}(1-D)^{2}-1-\frac{(u+v+2)(1+D)}{2}-\frac{(u+v+uv+1)(D+1)^{2}}{2(C-D)}\right)\nonumber\\
&+\left(-1
-\frac{[(u+v+2)]}{2}(D-1)\right)\nonumber\\
&\geq\left((D-C)-(u+v+2)D-\frac{(u+v+uv+1)(1-D^{2})}{C-D}\right)^{2}.
\end{align}
another inequality
\begin{align}\label{eq1.6}
&2\left(\left((D-C)-(u+v+2)D-\frac{(u+v+uv+1)(1-D^{2})}{C-D}\right)^{2}\right)+\nonumber\\
   & \left(\frac{(u+v+uv+1)}{2(D-C)}(1-D)^{2}-1-\frac{(u+v+2)(1+D)}{2}-\frac{(u+v+uv+1)(D+1)^{2}}{2(C-D)}\right)\nonumber\\
&+\left(-1
-\frac{[(u+v+2)]}{2}(D-1)\right)\geq 0,
\end{align}
whenever
\begin{align}\label{eq1.7}
        &\left(\frac{(u+v+uv+1)}{2(C-D)}(1-D)^{2}-1-\frac{(u+v+2)(1+D)}{2}-\frac{(u+v+uv+1)(D+1)^{2}}{2(D-C)}\right)\nonumber\\
&+\left(-1
-\frac{[(u+v+2)]}{2}(D-1)\right)\nonumber\\
&\leq\left((D-C)-(u+v+2)D-\frac{(u+v+uv+1)(1-D^{2})}{C-D}\right)^{2}.\end{align}
If $0\notin \phi^{\prime}(\mathbb(D))$ and $0\notin \phi^{\prime\prime}(\mathbb(D))$, then $$1+z\frac{\phi^{\prime\prime}(z)}{\phi^{\prime}(z)} \prec \frac{1+Cz}{1-Dz}.$$
    \end{theorem}
\begin{proof}
    We define a function
$q: \mathcal{D}\rightarrow \mathbb{C}$
$$q(z)=\frac{(C-D)\phi^{\prime}(z)+(1-D)z\phi^{\prime\prime}(z)}{(C-D)\phi^{\prime}(z)-(1+D)z\phi^{\prime\prime}(z)},$$
Then $$\frac{z\phi^{\prime\prime}(z)}{\phi^{\prime}(z)}=\frac{(C-D)(q(z)-1)}{(q(z)+1)+D(q(z)-1)}.$$
Taking logarithmic derivative both sides,
\begin{align}\label{eq1.1}
&\frac{1}{z}+\frac{\phi^{\prime\prime\prime}(z)}{\phi^{\prime\prime}(z)}-\frac{\phi^{\prime\prime}(z)}{\phi^{\prime}(z)}=\frac{q^{\prime}(z)}{q(z)-1}-\frac{q^{\prime}(z)(D+1)}{(q(z)+1)+D(q(z)-1)}\nonumber\\
& = 1+\frac{z\phi^{\prime\prime\prime}(z)}{\phi^{\prime\prime}(z)}-\frac{z\phi^{\prime\prime}(z)}{\phi^{\prime}(z)}=\frac{zq^{\prime}(z)}{q(z)-1}-\frac{zq^{\prime}(z)(D+1)}{(q(z)+1)+D(q(z)-1)}.
\end{align}
From \eqref{eq1.1}, we get
\begin{align*}
& \frac{z\phi^{\prime\prime\prime}(z)}{\phi^{\prime\prime}(z)}=\frac{2zq^{\prime}(z)}{(q(z)-1)((q(z)+1)+D(q(z)-1))}-1+\frac{z\phi^{\prime\prime}(z)}{\phi^{\prime}(z)} \;\mbox{or}, \\
& \frac{z\phi^{\prime\prime\prime}(z)}{\phi^{\prime\prime}(z)}\frac{z\phi^{\prime\prime}(z)}{\phi^{\prime}(z)}=\frac{2(C-D)(q(z)-1)zq^{\prime}(z)}{(q(z)-1)((q(z)+1)+D(q(z)-1))^{2}}-\frac{(C-D)(q(z)-1)}{(q(z)+1)+D(q(z)-1)}\\
&+\frac{(C-D)^{2}(q(z)-1)^{2}}{((q(z)+1)+D(q(z)-1))^{2}}.
\end{align*}
Differentiating both sides the differential equation $z(1-z)\phi^{\prime\prime}(u,v;w;z)+[w-(u+v+1)z]\phi^{\prime}(u,v;w;z)-uv\phi(u,v;w;z)=0$ and followed by dividing by $\phi^{\prime}(z)$ and multiplying $z$, we have
\begin{align}\label{eq1.2}
&(1-2z)\phi^{\prime\prime}(z)+z(1-z)\phi^{\prime\prime\prime}(z)+[w-(u+v+1)z]\phi^{\prime\prime}(z)-(u+v+1)\phi^{\prime}(z)-uv\phi^{\prime}(z)\nonumber\\
&=z(1-z)\frac{z\phi^{\prime\prime\prime}(z)}{\phi^{\prime}(z)}+[w+1-(u+v+3)z]\frac{z\phi^{\prime\prime}(z)}{\phi^{\prime}(z)}-(u+v+uv+1)z\nonumber\\
&=(1-z)\frac{z\phi^{\prime\prime\prime}(z)}{\phi^{\prime\prime}(z)}\frac{z\phi^{\prime\prime}(z)}{{\phi^{\prime}}(z)}+[w+1-(u+v+3)z]\frac{z\phi^{\prime\prime}(z)}{\phi^{\prime}(z)}-(u+v+uv+1)z\nonumber\\
&=(1-z)(\frac{2(C-D)(q(z)-1)zq^{\prime}(z)}{(q(z)-1)((q(z)+1)+B(q(z)-1))^{2}}-\frac{(D-C)(q(z)-1)}{(q(z)+1)+D(q(z)-1)}\nonumber\\
&+\frac{(C-D)^{2}(q(z)-1)^{2}}{((q(z)+1)+D(q(z)-1)^{2}})+[1+w-(w+u+3)z]\frac{(C-D)(q(z)-1)}{(q(z)+1)+D(q(z)-1)}\nonumber\\
&-(u+v+uv+1)z\nonumber\\
&= \frac{2(C-D)(1-z)zq^{\prime}(z)}{(q(z)+1)+D(q(z)-1))^{2}}+\frac{(C-D)^{2}(1-z)(q(z)-1)^{2}}{(q(z)+1)+D(Q(z)-1))^{2}}-(u+v+uv+1)z\nonumber\\
&+[w-(u+v+2)z]\frac{(C-D)(q(z)-1)}{(q(z)+1)+D(q(z)-1)}\nonumber\\
&=(1-z)zq^{\prime}(z)+\frac{(C-D)}{2}(1-z)(q(z)^{2}+1-2q(z))+\frac{w-(u+v+2)z}{2}((q(z))^{2}-1\nonumber\\
&+D((q(z))^{2}+1-2q(z))-\frac{(u+v+uv+1)z}{2(C-D)}(q(z)^{2}(D+1)^{2}+(1-D)^{2}+2q(z)(1-D^{2}))\nonumber\\
&=(1-z)zq^{\prime}(z)\nonumber\\
&+\left(\frac{C-D}{2}(1-z)+\frac{[w-(u+v+2)z]}{2}(1+D)-\frac{(u+v+uv+1)}{2(C-D)}z(1+D)^{2}\right)q(z)^{2}\nonumber\\
&+\left[(C-D)(z-1)-[w-(u+v+2)z]D-\frac{u+v+uv+1}{C-D}(1-D^{2})z\right]q(z)\nonumber\\
&+\frac{C-D}{2}(1-z)+\frac{[w-(u+v+2)z]}{2}(D-1)-\frac{(u+v+uv+1)}{2(C-D)}z(1-D)^{2}=0.
\end{align}
Let us consider a function $\psi(q(z),zq^{\prime}(z);z)=F_{1}zq^{\prime}(z)+F_{2}(q(z))^{2}+F_{3}q(z)+F_{4}$, where
$F_{1}=1-z$, 
$$F_{2}=\left(\frac{C-D}{2}(1-z)+\frac{[w-(u+v+2)z]}{2}(1+D)-\frac{(u+v+uv+1)}{2(C-D)}z(1+D)^{2}\right),$$
$$F_{3}=\left[(C-D)(z-1)-[w-(u+v+2)z]D-\frac{u+v+uv+1}{C-D}(1-D^{2})z\right],$$
$$F_{4}=\frac{C-D}{2}(1-z)
+\frac{[w-(u+v+2)z]}{2}(D-1)-\frac{(u+v+uv+1)}{2(C-D)}z(1-D)^{2}.$$
With the help of \eqref{eq1.2}, $\Omega=0$ imply that $\psi(q(z),zq^{\prime}(z);z)\in \Omega$. Assume that
$G_{1}=\Re(F_{1})=1-x$, $G_{2}=$\\$$\Re(F_{2})=\frac{C-D}{2}(1-x)+\frac{w(1+D)}{2}-\frac{(u+v+2)x(1+D)}{2}-\frac{(u+v+uv+1)x(D+1)^{2}}{2(C-D)},$$
$$G_{3}=\Re(iF_{3})=(C-D)(-y)+(u+v+2)(-y)D+\frac{(u+v+uv+1)(1-D^{2})}{C-D}y,$$
$$G_{4}=\frac{C-D}{2}(1-x)
+\frac{[w-(u+v+2)x]}{2}(D-1)-\frac{(u+v+uv+1)}{2(C-D)}x(1-D)^{2}.$$
For $\sigma\leq -\frac{(1+\rho^{2})}{2}$, $\rho\in \mathbb{R}$.
\begin{align*}
&\Re\psi(i\rho,\sigma; z)=G_{1}\sigma+G_{2}(i\rho)^{2}+G_{3}\rho+G_{4}\\
& = G_{1}\sigma-G_{2}(\rho)^{2}+G_{3}\rho+G_{4}\\
&\leq G_{1}\left(\frac{-(1+\rho^{2})}{2}\right)-G_{2}\rho^{2}+G_{3}\rho+G_{4}\\
&\leq \frac{-G_{1}}{2}-\frac{G_{1}\rho^{2}}{2}-G_{2}\rho^{2}+G_{3}\rho+G_{4}\\
&=\frac{-G_{1}-G_{1}\rho^{2}-2G_{2}\rho^{2}+2G_{3}\rho+2G{4}}{2}\\
&=\frac{-1}{2}[(G_{1}+2G_{2})\rho^{2}-2G_{3}\rho+G_{1}-2G_{4}]=Q(\rho).
\end{align*}
By given hypothesis \eqref{eq1.3} $(G_{1}+2G_{2})>0$, $Q(\rho)$ has maximum value at $\rho=\frac{G_{3}}{G_{1}+2G_{2}}$. Now finding value of $Q(\rho)$ at $\rho=\frac{G_{3}}{G_{1}+2G_{2}}$ , for all $ \rho, |x|,|y|<1$, we have
\begin{align*}
&Q(\rho)=\frac{-1}{2}[(G_{1}+2G_{2})\left(\frac{G_{3}}{G_{1}+2G_{2}}\right)^{2}-2G_{3}\left(\frac{G_{3}}{G_{1}+2G{2}}\right)-2G_{4}+G_{1}]<0\\
&=\frac{-1}{2}\left[\frac{-(G_{3})^{2}}{G_{1}+2G_{2}}-2G_{4}+G_{1}\right]\leq 0\\
&=\frac{G_{3}^{2}}{G_{1}+2G_{2}}\leq G_{1}-2G_{4}\\
&=G_{3}^{2}\leq (G_{1}+2G_{2})(G_{1}-2G_{4}).
\end{align*}
For $|x|<1,|y|<1$ and $y^{2}<1-x^{2}$ for above inequality, we have to prove

\begin{align*}
&((C-D)(-y)+(u+v+2)(-y)D+\frac{(u+v+uv+1)(1-D^{2})}{C-D}y)^{2} \\
&\leq (1-x+2(\frac{C-D}{2}(1-x)+\frac{w(1+D)}{2}-\frac{(u+v+2)x(1+D)}{2}\\
&-\frac{(u+v+uv+1)x(D+1)^{2}}{2(C-D)})(1-x-2(\frac{C-D}{2}(1-x)\\
&+\frac{[w-(u+v+2)x]}{2}(D-1)-\frac{(u+v+uv+1)}{2(C-D)}x(1-D)^{2})\\
\implies &((D-C)-(u+v+2)D+\frac{(u+v+uv+1)(1-D^{2})}{C-D})^{2}(1-x^{2}) \\
&\leq (1-x+2(\frac{C-D}{2}(1-x)+\frac{w(1+D)}{2}-\frac{(u+v+2)x(1+D)}{2}\\
&-\frac{(u+v+uv+1)x(D+1)^{2}}{2(C-D)}))(1-x-2(\frac{C-D}{2}(1-x)\\
&+\frac{[w-(u+v+2)x]}{2}(D-1)-\frac{(u+v+uv+1)}{2(C-D)}x(1-D)^{2})).
\end{align*}
The above inequality yields\\
$H(x)=h_{1}(C,D,O)x^{2}+h_{2}(C,D,O)x+h_{3}(C,D,O)\geq 0$, where
\begin{align*}
    &h_{1}(C,D,O)=\left((D-C)-(u+v+2)D-\frac{(u+v+uv+1)(1-D^{2})}{D-C}\right)^{2},\\
    &h_{2}(C,D,O)\\
    &=\left(\frac{(u+v+uv+1)}{2(C-D)}(1-D)^{2}-1-\frac{(u+v+2)(1+D)}{2}-\frac{(u+v+uv+1)(D+1)^{2}}{2(C-D)}\right)\\
&+\left(-1
-\frac{[(u+v+2)]}{2}(D-1)\right),\\
    &h_{3}(C,D,O)=2+\frac{w(1+D)}{2}+\frac{w(D-1)}{2}\nonumber\\
    &-\left((D-C)-(u+v+2)D+\frac{(u+v+uv+1)(1-D^{2})}{C-D}\right)^{2}.
\end{align*}
To obtain \eqref{eq1.3}, in first case \eqref{eq1.5} holds. Then $|\frac{-h_{2}}{2h_{1}}|=|\rho|\leq 1$ and therefore $H^{\prime}(x)=0$ at $x=\rho\in (-1,1)$. Since $H^{\prime\prime}(x)>0$, then minimum value at $x=\frac{-h_{2}}{2h_{1}}$ and from \eqref{eq1.4} it satisfies that
$H(x)\geq H(\rho)=h_{3}-\frac{{h_{2}}^{2}}{2h_{1}}\geq 0$.
Now, let \eqref{eq1.7} be hold. In this case $|\rho|\geq 1$ and $H^{\prime}(x)=2h_{1}(x)+h_{2}\leq 2h_{1}+h_{2}\leq 0$ equivalent to \eqref{eq1.6}, $H(x)$
is decreasing. Hence $H(x)\geq H(1)=h_{1}+h_{2}+h_{3}\geq 0$.
By help of Lemma \ref{lem2}, $\Re(p(z))>0$, it is equivalent to \\
$$q(z)=\frac{(C-D)\phi^{\prime}(z)+(1-D)z\phi^{\prime\prime}(z)}{(C-D)\phi^{\prime}(z)-(1+D)z\phi^{\prime\prime}(z)} \prec \frac{1+z}{1-z}.$$
Using definition of subordination, $\exists$ a self-map $l(z)$ such that $l(0)=0$ and
$$q(z)=\frac{(C-D)\phi^{\prime}(z)+(1-D)z\phi^{\prime\prime}(z)}{(C-D)\phi^{\prime}(z)-(1+D)z\phi^{\prime\prime}(z)} = \frac{1+l(z)}{1-l(z)}.$$
After simplifying above equation, we get
$$1+z\frac{\phi^{\prime\prime}(z)}{\phi^{\prime}(z)} = \frac{1+Cl(z)}{1-Dl(z)}$$
and therefore $$1+z\frac{\phi^{\prime\prime}(z)}{\phi^{\prime}(z)} \prec \frac{1+Cz}{1-Dz}.$$
\end{proof}
Let us consider the relation, we have
$$\frac{z(z\phi(u,v;w;z))^{\prime}}{z\phi(u,v;w;z)}=1+z\frac{\phi^{\prime\prime}(u-1,v-1;w-1;z)}{\phi^{\prime}(u-1,v-1;w-1;z)}.$$
Simultaneously with the Theorem \ref{thm3}, yields the following corollary as $z\phi(u,v;w;z)\in S^{*}[C,D].$

\begin{corollary}
    Assume that parameters $u,v$ and $w\neq 0,-1,-2,\cdots$ be constrained such that it satisfy the conditions  for $-1\leq D<C\leq 1$,
    \begin{align*}
    &1+C-D+(w-1)(1+D)\nonumber\\
    &-\left|1+C-D+(c+d+2)(1+D)+\frac{(u-1+v+(u-1)(v-1))(D+1)^{2}}{C-D}\right|>0,
    \end{align*}
    further
    \begin{align*}
        &\left(2+\frac{(w-1)(1+D)}{2}+\frac{(w-1)(D-1)}{2}-\left((D-C)-(u+v)D+\frac{(u+v+(u-1)(v-1)-1)(1-D^{2})}{D-C}\right)^{2}\right)\nonumber\\&2\left((D-C)-(u+v)D-\frac{(u-1+v+(u-1)(v-1))(1-D^{2})}{C-D}\right)^{2}+\nonumber\\&(\frac{(u-1+v+(u-1)(v-1))}{2(C-D)}(1-D)^{2}-1-\frac{(u+v)(1+D)}{2}-\frac{(u-1+v+(u-1)(v-1))(D+1)^{2}}{2(D-C)}+\nonumber\\
&\left(-1
-\frac{[(u+v)]}{2}(D-1)\right))^{2}\geq 0,
    \end{align*}
    whenever,
    \begin{align*}
        &\left(\frac{(u-1+v+(u-1)(v-1))}{2(C-D)}(1-D)^{2}-1-\frac{(u+v)(1+D)}{2}-\frac{(u-1+v+(u-1)(v-1))(D+1)^{2}}{2(C-D)}\right)\nonumber\\
&+\left(-1
-\frac{[(u+v)]}{2}(D-1)\right)\nonumber\\
&\geq\left((D-C)-(u+v)D-\frac{(u-1+v+(u-1)(v-1))(1-D^{2})}{C-D}\right)^{2}.
\end{align*}
another inequality
\begin{align*}
&2\left(\left((D-C)-(u+v)D-\frac{(u-1+v+(u-1)(v-1))(1-D^{2})}{C-D}\right)^{2}\right)+\nonumber\\
   & \left(\frac{(u-1+v+(u-1)(V-1))}{2(D-C)}(1-D)^{2}-1-\frac{(u+v)(1+D)}{2}-\frac{(u-1+v+(u-1)(v-1))(D+1)^{2}}{2(C-D)}\right)\nonumber\\
&+\left(-1
-\frac{[(u+v)]}{2}(D-1)\right)\geq 0,
\end{align*}
whenever
\begin{align*}
        &\left(\frac{(u-1+v+(u-1)(v-1))}{2(C-D)}(1-D)^{2}-1-\frac{(u+v)(1+D)}{2}-\frac{(u-1+v+(u-1)(V-1))(D+1)^{2}}{2(D-C)}\right)\nonumber\\
&+\left(-1
-\frac{[(u+v)]}{2}(D-1)\right)\nonumber\\
&\leq\left((D-C)-(u+v)D-\frac{(u-1+V+(U-1)(V-1))(1-D^{2})}{C-D}\right)^{2}.\end{align*}
If $0\notin \phi^{\prime}(\mathbb(D))$ and $0\notin \phi^{\prime\prime}(\mathbb(D))$, then $$\frac{z(z\phi(u,v;w;z))^{\prime}}{z\phi(u,v;w;z)} \prec \frac{1+Cz}{1-Dz}.$$
    \end{corollary}
Which imply Gauss hypergeometric function hold janowski starlikeness, i.e $z\phi(u,v;w;z)\in S^{*}[C,D].$

\section{conclusion}
The Investigation successfully established new and significant conditions for the normalized Gauss hypergeometric function, determining its geometric properties concerning exponential starlikeness, convexity, Janowski convexity, and starlikeness based on specific parameter values. The implications of these results were highlighted through various consequences underscoring the importance and novelty of these established conditions. These contributions add valuable insights to the understanding of the Gauss hypergeometric function's behavior within the unit disk, enriching the field of mathematical analysis. The result obtained in this article can be applied in further studies connected to fractional calculusas it csn be seen in recently published paper related to confluent hypergeometric funcion \cite{lupas}, Gaussian hypergeometric function \cite{ors} and other generalied hypergeometric function. Quantum calculus aspects can also be involved with Gaussian hypergeometric function motivated by result which is available in \cite{ gour} Moreover one open problem can be based for this function as: after holding which criteria it satisfies lemniscate starlikeness and convexity using subordination. 

{}

\end{document}